\documentclass[11pt,oneside,reqno]{amsart}

\usepackage{amsmath,amsthm}    
\usepackage{enumerate}
\numberwithin{equation}{section}
\theoremstyle{definition}
\newtheorem{Definition}{Definition}[section]
\newtheorem{Example}[Definition]{Example}
\newtheorem{Remark}[Definition]{Remark}

\theoremstyle{plain}
\newtheorem{Theorem}[Definition]{Theorem}
\newtheorem{Proposition}[Definition]{Proposition}

\newtheorem{Lemma}[Definition]{Lemma}

\usepackage{amssymb}        
\usepackage{mathrsfs}       
\usepackage{eucal}          
\usepackage{stmaryrd}       

\usepackage{hyperref}

\usepackage{geometry}

\usepackage[usenames,dvipsnames]{xcolor}
\usepackage{subfig}         

\newcommand{\la}{\lambda}

\newcommand{\N}{\mathbb{N}}
\newcommand{\Z}{\mathbb{Z}}
\newcommand{\Q}{\mathbb{Q}}

\newcommand{\C}{\mathbb{C}}

\newcommand{\Fh}{\mathfrak{h}}

\DeclareMathOperator{\End}{End}

\DeclareMathOperator{\Id}{Id}

\renewcommand{\hat}{\widehat}

\newcommand{\nord}{\begin{subarray}{c}\circ\\\circ\end{subarray}}

\title[Simple Whittaker modules over free bosonic orbifold VOAs]{Simple Whittaker modules over free bosonic orbifold vertex operator algebras}
\author{Jonas T. Hartwig\and Nina Yu}
\date{\today}

\address{Department of Mathematics, Iowa State University, Ames, IA-50011, USA}
\email{jth@iastate.edu}
\address{School of Mathematical Sciences, Xiamen University, Fujian, 361005, CHINA}
\email{ninayu@xmu.edu.cn}

\begin{document}
\maketitle
\begin{abstract}
We construct weak (i.e. non-graded) modules over the vertex operator algebra $M(1)^+$, which is the fixed-point subalgebra of the higher rank free bosonic (Heisenberg) vertex operator algebra with respect to the $-1$ automorphism. These weak modules are constructed from Whittaker modules for the higher rank Heisenberg algebra. We prove that the modules are simple as weak modules over $M(1)^+$ and calculate their Whittaker type when regarded as modules for the Virasoro Lie algebra. Lastly, we show that any Whittaker module for the Virasoro Lie algebra occurs in this way. These results are a higher rank generalization of some results by Tanabe \cite{T1}.
\end{abstract}


\section{Introduction}

Whittaker modules were first studied by Kostant \cite{K} in the context of finite-dimensional complex semisimple Lie algebras and play an important role in representation theory and geometry. They have since been defined and investigated  in many different settings such as the Virasoro Lie algebra \cite{OW,GLZ,MZ}, affine Lie algebras \cite{ALZ}, a general framework for Lie algebras \cite{BM}, and generalized Weyl algebras \cite{BO,FH}.

Vertex operator algebras (VOAs) \cite{FLM,LL} form a mathematical framework for studying conformal field theory. One of the fundamental methods for obtaining new VOAs from known ones is to take the fixed-point sub-VOA with respect to a finite order automorphism. Namely,  let $V$ be a rational VOA and $G$ be a finite automorphism group of $V$. Then $V^G$  is the resulting VOA is called an orbifold VOA because it corresponds to the conformal field theory on an orbifold (quotient) space.

Let $V$ be a rational VOA and $G$  be a finite automorphism group of $V$. The orbifold conjecture says that $V^G$ is rational and every irreducible $V^G$ module occurs in an irreducible $g$-twisted $V$-module for some $g\in G$. Miyamoto and Carnahan prove that if $G$ is solvable, then $V^G$ is rational \cite{CM}.  Later it is proved that every irreducible $V^G$-module occurs as an irreducible $g$-twisted $V$-module for some $g\in G$ \cite{DRX}.

One of the most important VOA is the Heisenberg VOA, also known as the free bosonic VOA and is denoted by $M(1)$. It is constructed from a complex vector space $\mathfrak{h}$  with a non-degenerate bilinear form. The dimension of the vector space is called the rank of the VOA. The VOA $M(1)$ has a natural automorphism of order two, induced by multiplication by $-1$ on the vector space $\mathfrak{h}$. The corresponding orbifold VOA is denoted by $M(1)^+$. Irreducible modules for $M(1)^+$ were classified by Dong and Nagatomo \cite{DN}.

 The orbifold conjecture still makes sense for non-$\mathbb{N}$-graded weak modules. The conjecture was confirmed by Tanabe for a class of simple non-$\mathbb{N}$-graded weak $M(1)^+$-modules defined by using Whittaker vectors for the Virasoro algebra \cite{T1}. Namely, let $M(1)$ be the VOA with the Virasoro element $\omega$ associated to the Hesenberg algebra of rank 1. Tanabe proved that any simple weak $M(1)^+$-module with at least one Whittaker vector is isomorphic to some simple weak $M(1)$-module or to some $ \theta$-twisted simple weak $M(1)$-module.

The structure of such Whittaker modules for $M(1)^+$ becomes more complicated when we consider higher rank case. In this paper, we construct a class of simple weak $M(1)^+$-modules when $M(1)$ is the VOA with the Virasoro element $\omega$ associated to the Heisenberg algebra of rank $\ell$ where $\ell\ge 1$. These modules are generalizations of Tanabe's modules and they are Whittaker modules with respect to the Virasoro algebra. We also determine explicitly the subfamily of all constructed simple weak $M(1)^+$-modules that give rise to the same Virasoro algebra module. In the rank one case there are only two, while in our case we get an affine variety which we explicitly describe.

\section{The vertex operator algebra $M(1)$}

\subsection{Definition of vertex operator algebra}

We recall the definition of a vertex operator algebra \cite{LL}.

\begin{Definition}
A \emph{vertex operator algebra} $(V,Y,\boldsymbol{1},\omega)$ consists of
\begin{enumerate}[{\rm (i)}]
\item a $\Z$-graded vector space $V=\bigoplus_{n\in\Z} V_{(n)}$ called the \emph{state space},
\item a linear map
\[
\begin{aligned}
Y(\cdot,z):V &\to \End(V)[[z,z^{-1}]] \\
 v &\mapsto Y(v,z)=\sum_{n\in\Z} v_n z^{-n-1}
\end{aligned}
\]
called the \emph{state-field correspondence};
\item a distinguished vector $\boldsymbol{1}\in V_{(0)}$ called the \emph{vacuum vector};
\item a distinguished vector $\omega\in V_{(2)}$ called the \emph{conformal vector};
\end{enumerate}
such that the following properties hold:
\begin{enumerate}
\item \emph{grading restrictions}: $\dim V_{(n)}<\infty$ for $n\in\Z$ and $V_{(n)}=0$ for $n\ll 0$;
\item \emph{truncation condition}: $Y(u,z)v\in V((z))$ for all $u,v\in V$;
\item \emph{vacuum property}: $Y(\boldsymbol{1},z)=1$;
\item \emph{creation property}: $Y(v,z)\boldsymbol{1}\in v+ V[[z]]z$ for all $v\in V$;
\item \emph{Jacobi identity}: For all $a,b\in V$,
\[
\begin{aligned}
 & z_{0}^{-1}\delta\left(\frac{z_{1}-z_{2}}{z_{0}}\right)Y\left(a,z_{1}\right)Y\left(b,z_{2}\right)-z_{0}^{-1}\delta\left(\frac{z_{2}-z_{1}}{-z_{0}}\right)Y\left(b,z_{2}\right)Y\left(a,z_{1}\right)\\
 & =z_{2}^{-1}\delta\left(\frac{z_{1}-z_{0}}{z_{2}}\right)Y\left(Y\left(a,z_{0}\right)b,z_{2}\right).
\end{aligned}
\]
\item \emph{Virasoro algebra relations}: There exists $c_V\in \C$ such that for all $m,n\in \Z$:
\begin{equation}
[L(m),L(n)]=(m-n)L(m+n)+\frac{1}{12}(m^3-m)\delta_{m+n,0}c_V
\end{equation}
where $L(n)=\omega_{n+1}$ for $n\in\Z$;
\item \emph{$L(0)$-grading:} $L(0)v=nv$ for all $v\in V_{(n)}$ and all $n\in\Z$;
\item \emph{$L(-1)$-derivation property}: $Y(L(-1)v,z)=\frac{d}{dz}Y(v,z)$ for all $v\in V$.
\end{enumerate}

\end{Definition}

\subsection{Definition of the vertex operator algebra $M(1)$}

We recall the well-known construction of the vertex operator algebra $M(1)$, called the \emph{Heisenberg} (or \emph{free bosonic}) vertex operator algebra.

Let $\mathfrak{h}$ be the complexification of an $\ell$-dimensional Euclidean space with orthonormal basis $\{h_1,h_2,\ldots,h_\ell\}$, viewed as an abelian Lie algebra. Its affinization is
\begin{equation}
\hat{\mathfrak{h}} = \mathfrak{h}\otimes\C[t,t^{-1}]\oplus\C K 
\end{equation}
with
\begin{equation} \label{eq:heisenberg}
[K,\hat{\mathfrak{h}}]=0,\qquad 
[a(m),b(n)]=m\delta_{m+n,0}(a,b)K
\end{equation}
for $a,b\in\mathfrak{h}$, $m,n\in\Z$ and we put $a(n)=a\otimes t^n$. We use the form $(\cdot,\cdot)$ to identify $\mathfrak{h}$ with its dual space $\mathfrak{h}^*$.
Let $\C e^0$ be the one-dimensional module over the Lie algebra $\mathfrak{h}\otimes\C[t]\oplus\C K$ with action given by
\[ h(n)e^0=0,\quad  \forall h\in \mathfrak{h},\, n\ge 0;\qquad  Ke^0 = e^0. \]
Define the vector space $M(1)$ by
\begin{equation}
M(1) = U(\hat{\mathfrak{h}})\otimes_{U(\mathfrak{h}\otimes\C[t]\oplus\C K)} \C e^0.
\end{equation}
On $M(1)$ define the state-field correspondence by
\begin{equation}
Y(a^{(1)}(n_1)\cdots a^{(r)}(n_r)e^0,z)=a^{(1)}(z)_{n_1}\cdots a^{(r)}(z)_{n_r}\Id_{M(1)}
\end{equation}
for $a^{(i)}\in\mathfrak{h}$ and $n_i\in\Z$, (see \cite[Thm.~6.2.11]{LL} for details).
The vacuum vector is $\boldsymbol{1}=e^0$ and the conformal vector is given by
\begin{equation} \label{eq:virasoro-vector}
\omega = \frac{1}{2}\sum_{i=1}^\ell h_i(-1)^2\boldsymbol{1}.
\end{equation}
In particular,
\begin{equation}
Y(\omega,z) = L(z) = \sum_{n\in\Z} L_n z^{-n-2},\qquad
L_n=\frac{1}{2}\sum_{m\in\Z}\sum_i \nord h_i(-m)h_i(m+n) \nord
\end{equation}

\section{The fixed-point vertex operator subalgebra $M(1)^+$}

Consider the order two automorphism $\theta$ of the vector space $M(1)$ given by
\begin{equation}\label{eq:theta}
\theta\big( h_{i_1}(-n_1)h_{i_2}(-n_2)\cdots h_{i_k}(-n_k) \boldsymbol{1}\big) = (-1)^k h_{i_1}(-n_1)h_{i_2}(-n_2)\cdots h_{i_k}(-n_k) \boldsymbol{1}
\end{equation}
for $i_j\in\{1,2,\ldots,\ell\}$ for all $j$ and $n_1\ge n_2\ge \cdots\ge n_k>0$.
 Let $M(1)^+$ be the corresponding subspace of fixed-points with respect to $\theta$:
\begin{equation}
M(1)^+=\big\{ v\in M(1)\mid \theta(v)=v\big\}.
\end{equation}
Note that the Virasoro vector $\omega$ belongs to $M(1)^+$, by \eqref{eq:virasoro-vector}. It is well-known that $M(1)^+$ is a vertex operator subalgebra of $M(1)$.
Let
\begin{equation}
J_a= h_a(-1)^4\boldsymbol{1}-2h_a(-3)h_a(-1)\boldsymbol{1}+\frac{3}{2}h_a(-2)^2\boldsymbol{1},\qquad a=1,2,\ldots,\ell.
\end{equation}
The following statement is known, see \cite{DN}.

\begin{Proposition}
\begin{enumerate}[{\rm (a)}]
\item The VOA $M(1)^+$ is generated by the the following subset:
\begin{equation}
\{J_a\mid a=1,2,\ldots,\ell\}\cup\{h_a(-1)h_b(-1)\boldsymbol{1}\mid a,b=1,2,\ldots,\ell\}
\end{equation}
\item The Zhu algebra $A(M(1)^+)$ is finitely generated.
\end{enumerate}
\end{Proposition}

\section{The weak $M(1)^+$-modules $M(1,\boldsymbol{\lambda})$ and  $M(1,\boldsymbol{\lambda})(\theta)$}

\subsection{Definition of weak $g$-twisted $V$-modules}

\begin{Definition} \cite{AbeDongLi}
Let $V$ be a vertex operator algebra and $g$ be an automorphism of $V$ of finite order $T$ with polarization $V=\bigoplus_{r=0}^{T-1} V^r$, $V^r=\{v\in V\mid g(v)=e^{-\frac{r}{T}2\pi i}v\}$. A \emph{weak $g$-twisted $V$-module} is a pair $(M,Y_M)$ where $M$ is a vector space and $Y_M:V\to\End(M)[[z,z^{-1}]], v\mapsto \sum_{n\in\Z} v_nz^{-n-1}$ is a linear map such that for $0\le r\le T-1$, $a\in V^r$, $b\in V$ and $u\in M$:
\begin{enumerate}[{\rm (i)}]
\item $Y_M(a,z)u\in z^{-\frac{r}{T}}M((z))$,
\item $Y_M(\boldsymbol{1},z)=\Id_M$,
\item the twisted Jacobi identity holds:
\begin{alignat*}{1}
 & z_{0}^{-1}\delta\left(\frac{z_{1}-z_{2}}{z_{0}}\right)Y_{M}\left(a,z_{1}\right)Y_{M}\left(b,z_{2}\right)-z_{0}^{-1}\delta\left(\frac{z_{2}-z_{1}}{-z_{0}}\right)Y_{M}\left(b,z_{2}\right)Y_{M}\left(a,z_{1}\right)\\
 & =z_{2}^{-1}\left(\frac{z_1-z_0}{z_2}\right)^{-\frac{r}{T}}\delta\left(\frac{z_{1}-z_{0}}{z_{2}}\right)Y_{M}\left(Y\left(a,z_{0}\right)b,z_{2}\right).
\end{alignat*}
\end{enumerate}
When $g=1$ a weak $g$-twisted $V$-module is called a \emph{weak $V$-module}. A \emph{$g$-twisted weak $V$-submodule} of a $g$-twisted weak $V$-module $M$ is a subspace $N$ of $M$ such that $a_n N\subseteq N$ for all $a\in V$ and $n\in\Q$. If $M\neq 0$ and $M$ has no weak $g$-twisted $V$-submodules except $0$ and $M$ then $M$ is said to be \emph{simple}.
\end{Definition}

\subsection{Construction of $M(1,\boldsymbol{\lambda})$}

Let $\boldsymbol{\lambda}=(\lambda_0,\lambda_1,\ldots,\lambda_r)$ be a sequence where $\la_i\in \mathfrak{h}$ such that $\lambda_n=0$ for $n\gg 0$.
Let $\mathbb{C}e^{\boldsymbol{\lambda}}$ be the one-dimensional module over the Lie algebra $\mathfrak{h}\otimes \C[t]\oplus \C K$ with action given by
\[ h(n)e^{\boldsymbol{\lambda}}=(h,\lambda_n)e^{\boldsymbol{\lambda}},\; h\in \mathfrak{h},\, n\ge 0;\quad  Ke^{\boldsymbol{\lambda}} = e^{\boldsymbol{\lambda}}. \]
Consider the corresponding induced $U(\widehat{\mathfrak{h}})$-module
\begin{equation}
M(1,\boldsymbol{\lambda})=U(\hat{\mathfrak{h}})\otimes_{U(\mathfrak{h}\otimes\C[t]\oplus\C K)} \C{e^{\boldsymbol{\lambda}}}.
\end{equation}
$M(1,\boldsymbol{\lambda})$ is an example of a \emph{Whittaker module} because it is generated as a left $U(\hat{\mathfrak{h}})$-module by a \emph{Whittaker vector} $e^\la$, which is by definition a common eigenvector for $\mathfrak{h}\otimes\C[t]\oplus\C K$ such that $(\mathfrak{h}\otimes t^n\C[t])e^\la=0$ for $n\gg 0$.
By the PBW theorem, $M(1,\boldsymbol{\lambda})$ has a basis consisting of monomials
\begin{equation}\label{eq:monomials}
h_{i_1}(-n_1)h_{i_2}(-n_2)\cdots h_{i_k}(-n_k) e^\la
\end{equation}
where $i_j\in\{1,2,\ldots,\ell\}$ for all $j$ and $n_1\ge n_2\ge \cdots\ge n_k>0$.

\begin{Example}\label{ex:r0case}
Suppose $\boldsymbol{\lambda}=(\la_0,0,0,\ldots)$ for some $\la_0\in\mathfrak{h}$. Then $M(1,\boldsymbol{\lambda})=M(1,\la_0)$ is a highest weight $\hat{\mathfrak{h}}$-module, the level $1$ Verma module of highest weight $\la_0$.
\end{Example}

\begin{Example}
Suppose $\boldsymbol{\lambda}=(\la_0,\la_1,0,\ldots)$ for some $\la_0,\la_1\in\mathfrak{h}$ where $\la_1\neq 0$. Then for $|n|>1$ we have the usual behavior:
\[h(n) e^{\boldsymbol{\lambda}} =0, \forall n>1,\; h\in\mathfrak{h}\]
while the set of vectors
\[h_{i_1}(-n_1)h_{i_2}(-n_2)\cdots h_{i_k}(-n_k)e^{\boldsymbol{\lambda}}
\quad\text{where $n_1\ge n_2\ge \cdots \ge n_k>1$ and $1\le i_j\le \ell$}\] are linearly independent.
Also $h(0)e^{\boldsymbol{\lambda}}=(h,\la_0)e^{\boldsymbol{\lambda}}$ for $h\in\mathfrak{h}$.
However, for $n=\pm 1$ something different happens:
\begin{equation}\label{eq:whittaker-n1}
h(1) e^{\boldsymbol{\lambda}} = (h,\la_1) e^{\boldsymbol{\lambda}} ,\; \forall h\in  \mathfrak{h}.
\end{equation}
The Heisenberg relation \eqref{eq:heisenberg} implies that
\[h(1)h(-1) - h(-1)h(1) = (h,h)K,\;\forall h\in \mathfrak{h}.\]
Acting on $e^{\boldsymbol{\lambda}}$ in both sides and using \eqref{eq:whittaker-n1} we get
\[h(1)h(-1) e^{\boldsymbol{\lambda}} = \big( (h,\la_1) h(-1) + (h,h)\big) e^{\boldsymbol{\lambda}}. \]
Now choose $h\in\mathfrak{h}$ such that $(h,\la_1)\neq 0$ (possible since $(\cdot,\cdot)$ is non-degenerate and $\la_1\neq 0$). Then we see that $M(1,\boldsymbol{\lambda})$ is not an $\N$-graded module: When acting by $h(1)$ on the homogeneous vector $h(-1)e^{\boldsymbol{\lambda}}$ we obtain a sum of two homogeneous vectors of different degrees:
\[h(1)\cdot h(-1)e^{\boldsymbol{\lambda}} = (h,\la_1)h(-1)e^{\boldsymbol{\lambda}}  + (h,h)e^{\boldsymbol{\lambda}}. \]
Also note that this means that $h(1)$ does not even act locally nilpotently: For any $N>0$ we have
\[h(1)^N\cdot h(-1)e^{\boldsymbol{\lambda}} \in  (h,\la_1)^N h(-1)e^{\boldsymbol{\lambda}}  + \C e^{\boldsymbol{\lambda}}. \]
\end{Example}

\subsection{Construction of $M(1,\boldsymbol{\la})(\theta)$}

The following result is well known \cite{FLM}.

\begin{Proposition}
Let $V$ be a vertex operator algebra, $g$ be an automorphism of $V$, and $M$ be a simple weak $g$-twisted $V$-module. Then $M$ is a simple weak module over the orbifold vertex operator algebra $V^g$.
\end{Proposition}

Let
\begin{equation}
\hat{\mathfrak{h}}[-1] = \mathfrak{h}\otimes t^{1/2}\C[t,t^{-1}]\oplus \C K
\end{equation}
with Lie algebra bracket
\begin{equation}
[K,\hat{\mathfrak{h}}[-1]]=0, \qquad
[a(m),b(n)]=m(a,b)\delta_{m+n,0}K, \qquad
a,b\in\mathfrak{h},\, m,n\in\frac{1}{2}+\Z,
\end{equation}
where $a(n)=a\otimes t^n$ for $a\in\mathfrak{h}$, $n\in\frac{1}{2}+\Z$.

Let $\boldsymbol{\la}=(\la_{1/2},\la_{3/2},\ldots)$ be a sequence of elements of $\mathfrak{h}$ such that $\lambda_n=0$ for $n\gg 0$, and
let $\C e^{\boldsymbol{\la}}$ be the one-dimensional $U(\mathfrak{h}\otimes t^{1/2}\C[t]\oplus\C K)$-module given by
\begin{align}
a(n) e^{\boldsymbol{\la}} &= (\la_n,a) e^{\boldsymbol{\la}}, \qquad a\in\mathfrak{h},\, n=\frac{1}{2},\frac{3}{2},\ldots\\
K e^{\boldsymbol{\la}} &= e^{\boldsymbol{\la}}.
\end{align}
Define the vector space $M(1,\boldsymbol{\la})(\theta)$ by
\begin{equation}
M(1,\boldsymbol{\la})(\theta) = U(\hat{\mathfrak{h}}[-1])\otimes_{U(\mathfrak{h}\otimes t^{1/2}\C[t]\oplus\C K)} \C e^{\boldsymbol{\la}}
\end{equation}
By the PBW theorem, $M(1,\boldsymbol{\la})(\theta)$ has a basis
\begin{equation}
\left\{h_{i_1}(-n_1)h_{i_2}(-n_2)\cdots h_{i_r}(-n_r)e^{\boldsymbol{\la}}\mid r\in\N, n_1\ge n_2\ge\cdots\ge n_r, i_j=1,2\ldots,\ell \right\}.
\end{equation}

Then $M(1,\boldsymbol{\la})(\theta)$ has the structure of a weak $\theta$-twisted $M(1)$-module as follows \cite{FLM}.
For $a\in\mathfrak{h}$, put
\begin{equation}
a(z)=\sum_{n\in \frac{1}{2}+\Z} a(n)z^{-n-1}
\end{equation}
and for $u=a_1(-n_1)a_2(-n_2)\cdots a_r(-n_r)e^{0} \in M(1)$, define

\begin{equation}
Y_0(u,z)=\nord \frac{1}{(n_1-1)!}\left(\frac{d}{dz}\right)^{n_1-1}\left(a_1(z)\right) \cdots \frac{1}{(n_r-1)!}\left(\frac{d}{dz}\right)^{n_r-1}\left(a_r(z)\right)\nord
\end{equation}
Let $c_{mn}\in\Q$ for $m,n\in\N$ be given by
\begin{equation}
\sum_{m,n=0}^\infty c_{mn}z^mw^n = -\log \frac{(1+z)^{1/2}+(1+w)^{1/2}}{2}
\end{equation}
and set
\begin{equation}
\Delta_z = \sum_{i=1}^\ell\sum_{m,n=0}^\infty c_{mn} h_i(m)h_i(n)z^{-m-n},
\end{equation}

\begin{equation}
Y_{M(1,\boldsymbol{\la})(\theta)}(u,z) = Y_0(e^{\Delta_z} u,z).
\end{equation}

\section{Bosonic Fock space realization of $M(1,\boldsymbol{\la})$ and $M(1,\boldsymbol{\la})(\theta)$}

\subsection{Untwisted case}

It will be useful to study the $M(1)^+$-module $M(1,\boldsymbol{\lambda})$ using a Whittaker analog of the bosonic Fock representation.
First, since $h_i(-m)$ and $h_j(-n)$ commute for all $i,j\in\{1,2,\ldots,\ell\}$ and all positive integers $m$ and $n$, there exists a $U(\widehat{\mathfrak{h}}_-)$-module isomorphism
\begin{equation}
\begin{aligned}
\varphi:M(1,\boldsymbol{\lambda}) &\to B=\C[x_{in}\mid i=1,2,\ldots,\ell;\; n=1,2,\ldots] \\
h_{i_1}(-n_1)h_{i_2}(-n_2)\cdots h_{i_k}(-n_k)e^{\boldsymbol{\lambda}} &\mapsto x_{i_1n_1}x_{i_2n_2}\cdots x_{i_kn_k}
\end{aligned}
\end{equation}
where $U(\widehat{\mathfrak{h}}_-)$ acts on $B$ via $h_i(-n).f = x_{in}f$.
Then there is a unique action of $U(\widehat{\mathfrak{h}})$ on $B$ such that $\varphi$ is a $U(\widehat{\mathfrak{h}})$-module isomorphism. Explicitly, it is given by

\begin{equation}
\left\{
\begin{aligned}
h_i(-n) f &= x_{in}f\qquad i=1,2,\ldots,\ell;\; n=1,2,\ldots \\
h_i(n) f &= \big(\partial_{in} + (\la_n, h_i)1\big)(f)\qquad i=1,2,\ldots,\ell;\; n=0,1,2,\ldots
\end{aligned}\right.
\end{equation}
for all $f\in B$, where $\partial_{in}=n\frac{\partial}{\partial x_{in}}$.
We denote $B$ by $B_{\boldsymbol{\la}}$ when regarded as a $U(\widehat{\mathfrak{h}})$-module in this way. In particular we have
\begin{equation}\label{eq:M1plus-action-on-B}
h_i(m)h_j(n)f = \big(\partial_{im}+(\la_m,h_i)1\big)\big(\partial_{jn}+(\la_n,h_j)1\big)(f)
\end{equation}
or, equivalently,
\begin{equation} \label{eq:himhjn-identity}
\big(h_i(m)h_j(n)-(\la_m,h_i)(\la_n,h_j)\big)f = (\la_m,h_i)\partial_{jn}(f) + (\la_n,h_j)\partial_{im}(f) + \partial_{im}\partial_{jn}(f)
\end{equation}
for all $f\in B_{\boldsymbol{\lambda}}$, $i,j=1,2,\ldots,\ell$ and positive integers $m$ and $n$.

Let $U(\widehat{\mathfrak{h}})^+$ be the subalgebra of $U(\widehat{\mathfrak{h}})$ generated by all quadratic elements $h_i(m)h_j(n)$ where $i,j=1,2,\ldots,\ell$ and $m,n\in\Z$.

\begin{Lemma} \label{lem:even-subalg}
$B_{\boldsymbol{\lambda}}$ is a simple $U(\widehat{\mathfrak{h}})^+$-module if and only if the weak $M(1)^+$-module $M(1,\boldsymbol{\lambda})$ is simple.
\end{Lemma}

\begin{proof}
Let $u\in M(1,\boldsymbol{\la})$ and let $m\in\N$ such that $h(i)u=0$ for all $h\in\Fh$ and all $i>m$.
Then, as in \cite{T1}, we have the identity
\begin{equation} \label{eq:binom1}
\begin{aligned}
&\big(h_a(-p-1)h_b(-q-1)\mathbf{1}\big)_{n+1}\,u \\
&\quad = \sum_{\substack{i+j\ge 0\\ i+j=2m-n}} \binom{i-m-1}{p}\binom{j-m-1}{q} \nord h_a(-i+m)h_b(-j+m)\nord \,u
\end{aligned}
\end{equation}
which holds for all non-negative integers $p,q$ and all $a,b\in\{1,2,\ldots,\ell\}$.
This shows that if $B'$ is a nonzero proper $U(\widehat{\mathfrak{h}})^+$-submodule of $B_{\boldsymbol{\lambda}}$, then $\varphi^{-1}(B')$ is a nonzero proper $M(1)^+$-submodule of $M(1,\boldsymbol{\la})$.

Conversely, let $W$ be any nonzero proper $M(1)^+$-submodule of $M(1,\boldsymbol{\la})$.
By \cite[Lemma 2.3]{T1} the coefficient matrix in \eqref{eq:binom1} is invertible and therefore $U(\widehat{\mathfrak{h}})^+\varphi(W)\subseteq \varphi(W)$ which shows that $\varphi(W)$ is a nonzero proper $U(\widehat{\mathfrak{h}})^+$-submodule of $B_{\boldsymbol{\lambda}}$.
\end{proof}

\subsection{Twisted case}

Analogously to the untwisted case, we find a bosonic fock space realization of the $U(\hat{\mathfrak{h}}[-1])$-module $M(1,\boldsymbol{\la})(\theta)$.

For $\boldsymbol{\la}=(\la_{1/2},\la_{3/2},\ldots)$, define
\begin{equation}
B_{\boldsymbol{\la}}^{\theta}=\C[x_{in}\mid i=1,2,\ldots,\ell; n\in\frac{1}{2}+\Z]
\end{equation}
equipped with the $U(\hat{\mathfrak{h}}[-1])$-module action
\begin{align}
Kf&=f, \\
h_i(-n)f &= x_{in}f, \\
h_i(n)f &= \left(\partial_{in}+(\la_n,h_i)1\right)(f),
\end{align}
for $i=1,2,\ldots,\ell$ and $n=\frac{1}{2},\frac{3}{2},\ldots$.

Then the linear map
\begin{equation}
\begin{aligned}
\psi:M(1,\boldsymbol{\la})(\theta) &\to B_{\boldsymbol{\la}}^{\theta} \\
h_{j_1}(-m_1)\cdots h_{j_r}(-m_r)e^{\boldsymbol{\la}} &\mapsto x_{j_1m_1}\cdots x_{j_rm_r}
\end{aligned}
\end{equation}
is an isomorphism of $U(\hat{\mathfrak{h}}[-1])$-modules.

In particular we have the identity

\begin{equation}
\left(h_i(m)h_j(n)-(\la_m,j_i)(\la_n,h_j)1\right)(f)=
\left(\partial_{im}\partial_{jm}+(\la_m,h_i)\partial_{jn}+(\la_n,h_j)\partial_{im}\right)(f)
\end{equation}

for $i,j\in\{1,2,\ldots,\ell\}$ and $m,n\in\frac{1}{2}+\Z$ and $f\in M(1,\boldsymbol{\la})(\theta)\cong B_{\boldsymbol{\la}}^{\theta}$.

Let $\nu$ be the involution on the Lie algebra $\hat{\mathfrak{h}}[-1]$ given by negation, and let $U(\hat{\mathfrak{h}}[-1])^+$ denote the corresponding fixed-point subalgebra.

\begin{Lemma}\label{lem:even-subalg-theta}
For any $\boldsymbol{\la}=(\la_{1/2},\la_{3/2},\ldots)$, the weak $M(1)^+$-module $M(1,\boldsymbol{\la})(\theta)$ is simple if and only if $B_{\boldsymbol{\la}}^{\theta}$ is an simple $U(\hat{\mathfrak{h}}[-1])^+$-module.
\end{Lemma}

\begin{proof}
Similar to the proof of Lemma \ref{lem:even-subalg}, using the twisted version of Borcherds identity.
\end{proof}

\section{Main theorem: Simplicity of $M(1,\boldsymbol{\lambda})$ and $M(1,\boldsymbol{\lambda})(\theta)$ as weak $M(1)^+$-modules.}

We now prove that the weak $M(1)^+$-modules $M(1,\boldsymbol{\lambda})$
and $M(1,\boldsymbol{\la})(\theta)$
are simple for any nonzero
$\boldsymbol{\lambda}$.

\begin{Theorem} \label{thm:simplicity}
For any non-negative integer $r$ and any sequence $\boldsymbol{\la}=(\la_0,\la_1,\ldots,\la_r,0,0,\ldots)$ (respectively $\boldsymbol{\la}=(\la_{1/2},\la_{3/2},\ldots,\la_{r+1/2},0,0,\ldots)$) of elements of $\mathfrak{h}$ with $\la_r\neq 0$ (respectively $\la_{r+1/2}\neq 0$), the weak $M(1)^+$-module $M(1,\boldsymbol{\lambda})$ (respectively $M(1,\boldsymbol{\la})(\theta)$) is simple.
\end{Theorem}

\begin{proof}
We treat both cases simultaneously.
Let $B=B_{\boldsymbol{\la}}$ (respectively $B=B_{\boldsymbol{\la}}^{\theta}$) and $U^+=U(\hat{\mathfrak{h}})^+$ (respectively $U^+=U(\hat{\mathfrak{h}}[-1])^+$).
Put $I=\{1,2,\ldots,\ell\}$ where $\ell=\dim\mathfrak{h}$.
By Lemmas \ref{lem:even-subalg} and \ref{lem:even-subalg-theta}, it suffices to prove that $B$ is a simple $U^+$-module.
We equip $B$ with the usual weight gradation where $\deg x_{in}=n$.
When $\partial_{im}$ acts on any monomial in the variables $\{x_{jn}\}$, the result is either zero or the degree has decreased by $m$.
Thus for any $a\in B$,
\begin{equation}
\deg \partial_{im}(a)\le\deg(a)-m
\end{equation}
with equality if and only if $x_{im}$ occurs in some leading monomial of $a$ ($\deg 0=-\infty$).

Let $a\in B$ be a nonzero arbitrary polynomial. Let $d=\deg a$.
We prove that $1\in U^+ a$ by induction on $d$.
It suffices to show that there exists at least one leading term in $a$ that is mapped to something nonzero under the differential operator. So we may without loss of generality assume that $a$ is a sum of monomials of the same degree, $d$.
Any set of distinct monomials, each containing the variable $x_{in}$, remains linearly independent when acted upon by $\partial_{in}$.

If $d=0$ then this is trivial since then $a$ is a nonzero complex number.

If $d>0$ then we show that there exists $(i,j,m,n)\in I\times I\times \Z_+\times\Z_+$ (respectively $(i,j,m,n)\in I\times I \times (\frac{1}{2}+\N)\times (\frac{1}{2}+\N)$)
such that
\begin{equation} \label{eq:pf1}
b=\big(h_i(m)h_j(n)-(\la_m,h_i)(\la_n,h_j)\big)a
\end{equation}
is nonzero. If we can show this, then since $\deg(b) < \deg(a)=d$ it follows that $1\in U^+ b$ by the induction hypothesis. Since
$h_i(m)h_j(n)-(\la_m,h_i)(\la_n,h_j)\in U^+$ we get
\[1\in U^+b \subseteq U^+ a\]
as desired.
To prove that we can find such $(i,j,m,n)$, note that by \eqref{eq:himhjn-identity}, the expression \eqref{eq:pf1} is equal to
\begin{equation}
(\la_m,h_i)\partial_{jn}(f)+(\la_n,h_j)\partial_{im}(f)+\partial_{im}\partial_{jn}(a).
\end{equation}
Since $a$ is a nonconstant polynomial in the variables $x_{im}$, there exists $(i,m)$ such that $\partial_{im}(a)\neq 0$.
Fix $m$ to be the smallest positive (half-)integer such that there exists $i_0\in I$ such that $\partial_{i_0m}(a)\neq 0$.
Furthermore, let $n$ be any non-negative (half-)integer such that $\la_n\neq 0$.
The case when $\la_n=0$ for all $n>0$ corresponds to the highest weight case, which is already known \cite[Proposition 2.2.2]{DN}. So we can assume that $n>0$.
We now consider three cases.

(Case 1): If $n>m$ then pick any $j_0\in I$ such that $(\la_n, h_{j_0})\neq 0$. This is possible since $\la_n\neq 0$, $h_j$ form a basis for $\mathfrak{h}$, and $(\cdot,\cdot)$ is non-degenerate.
Then with $(i,j,m,n)=(i_0,j_0,m,n)$ we get
\[(\la_n, h_{j_0}) \partial_{im}(a) + \text{lower degree terms}\]
which is has a nonzero leading term of degree $(\deg a) - m$.

(Case 2): If $n<m$ then again pick any $j_0\in I$ such that $(\la_n, h_{j_0})\neq 0$. Then with
 $(i,j,m,n)=(i_0,j_0,m,n)$ we get
\[(\la_n, h_{j_0}) \partial_{i_0m}(a) + (\la_m,h_i)\partial_{j_0n}(a) + \partial_{im}\partial_{j_0n}(a) = (\la_n, h_{j_0}) \partial_{i_0m}(a) + \partial_{im}\partial_{j_0n}(a)  \]
since $\partial_{j_0n}(a)=0$ by minimality of $m$. The term $\partial_{im}\partial_{j_0n}(a) $ has degree $(\deg a)-(n+m) < (\deg a)-m  = \deg \partial_{i_0m}(a)$, and the first term is nonzero by the choice of $(i_0,m)$ and since $(\la_n,h_{j_0})\neq 0$.

(Case 3): If $n=m$ there are two subcases:

(Case 3a) Suppose $(\la_m,h_{i_0})\neq 0$. Then take $(i,j,m,n)=(i_0,i_0,m,m)$ we get

\[2(\la_m,h_{i_0})\partial_{i_0m}(a) + \text{lower degree terms}\]

and the coefficient is nonzero.

(Case 3b) Suppose $(\la_m,h_{i_0})=0$. Then find any other $j_0\in I$ such that $(\la_m,h_{j_0})\neq 0$. This is possible since we are in the case where $m=n$ and $\la_n\neq 0$ by choice of $n$.
Then with $(i,j,m,n)=(i_0,j_0,m,n)$ we get
\[ (\la_m,h_{j_0})\partial_{i_0m}(a) + (\la_m,h_{i_0})\partial_{j_0m}(a) + \text{lower degree terms}=
(\la_m,h_{j_0})\partial_{i_0m}(a)  + \text{lower degree terms}\]
since $(\la_m,h_{i_0}) = 0$.

This covers all cases and proves the claim.
\end{proof}

\section{$M(1,\boldsymbol{\lambda})$ and $M(1,\boldsymbol{\lambda})(\theta)$ are $\omega$-Whittaker $M(1)^+$-modules}

\subsection{Definition of $\omega$-Whittaker modules}

There is no general definition of a Whittaker module for a vertex operator algebra. However, such a notion exists for Virasoro algebras. Following \cite{T1} we therefore make the following definitions.

\begin{Definition}
Let $V$ be a vertex operator algebra with conformal vector $\omega$ and let $M$ be a weak $V$-module.
\begin{enumerate}
\item A nonzero vector $v\in M$ is a \emph{Whittaker vector for $\omega$} if there exist a positive integer $r$ and a sequence of complex numbers $\zeta=(\zeta_{r+1},\zeta_{r+2},\ldots,\zeta_{2r+\epsilon})$ where $\epsilon\in\{0,1\}$ and $\zeta_{2r+\epsilon}\neq 0$ such that
\begin{equation}
\omega_i v = \begin{cases}
\zeta_i v,& j=r+1,\, r+2,\,\ldots,\, 2r+\epsilon\\
0,& i>2r+\epsilon.
\end{cases}
\end{equation}
The sequence $\zeta$ is called the \emph{type} of $v$.
\item We say that $M$ is an \emph{$\omega$-Whittaker module} if $M$ is generated as a weak $V$-module by a Whittaker vector for $\omega$.
\end{enumerate}
\end{Definition}

\subsection{$M(1,\boldsymbol{\lambda})$ and $M(1,\boldsymbol{\lambda})(\theta)$ are $\omega$-Whittaker $M(1)^+$-modules}

\begin{Proposition}
For any positive integer $r$ and any sequence $\boldsymbol{\la}=(\la_0,\la_1,\ldots,\la_r,0,0,\ldots)$ (respectively $\boldsymbol{\la}=(\la_{1/2},\la_{3/2},\ldots,\la_{r-1/2},0,0,\ldots)$) of elements of $\mathfrak{h}$ with $\la_r\neq 0$ (respectively $\la_{r-1/2}\neq 0$), the weak $M(1)^+$-module $M(1,\boldsymbol{\lambda})$ (respectively $M(1,\boldsymbol{\la})(\theta)$) is an $\omega$-Whittaker module.
\end{Proposition}

\begin{proof}
By Theorem \ref{thm:simplicity} these $M(1)^+$-modules are simple, hence in particular generated by $e^{\boldsymbol{\la}}$. Therefore it suffices to show that, in both cases, $e^{\boldsymbol{\la}}$ is a Whittaker vector for $\omega$. We first consider the case of $M(1,\boldsymbol{\la})$.
Since $m+n\ge r$ implies $\min\{m,n\}\ge 0$ if $\max\{m,n\}\le r$ we have for all $j=r+1,r+2,\ldots,2r+1$,
\begin{align*}
\omega_j e^{\boldsymbol{\lambda}} &= \frac{1}{2}\sum_{\substack{m,n\in\Z\\ m+n=j-1}} \sum_i \nord h_i(m)h_i(n)\nord \, e^{\boldsymbol{\lambda}}  \nonumber \\
&=\frac{1}{2}\sum_{\substack{m,n\in\N\\ m+n=j-1}} (\la_m,\la_n) e^{\boldsymbol{\lambda}} \in \C e^{\boldsymbol{\lambda}}
\label{eq:omega_j-on-vacuum}
\end{align*}
where we used that $\sum_i(\la_m,h_i)(\la_n,h_i)=(\la_m,\la_n)$. Taking $j=2r+1$ there is only one nonzero term in the sum, $\omega_{2r+1}e^{\boldsymbol{\lambda}}=\frac{1}{2}(\la_r,\la_r)e^{\boldsymbol{\lambda}}\neq 0$.
This shows that $M(1,\boldsymbol{\lambda})$ is a $\omega$-Whittaker module for $M(1)^+$.

Now consider the weak $M(1)^+$-module
$M(1,\boldsymbol{\la})(\theta)$. The action of the conformal vector is given by \cite[Example 18.9]{KRR}:

\begin{equation}
L^{\textrm{tw}}(z) = \frac{1}{2}\sum_{i=1}^\ell \nord h\,_i(z)^2 \,\nord +\frac{1}{16}z^{-2}
\end{equation}

This shows that
\begin{equation}
Y_{M(1,\boldsymbol{\la})(\theta)} (\omega,z) =  \sum_{n\in\Z} \omega_n z^{-n-1}
\end{equation}
where
\begin{equation}
\omega_j = \frac{1}{2}\sum_{\substack{m,n\in\Z+\frac{1}{2}\\ m+n=j-1}} \left(\frac{1}{16}\delta_{j,1}+\sum_{i=1}^\ell \nord\, h_i(m)h_i(n)\,\nord\right)
\end{equation}
Suppose that $j\in\{r+1,r+2,\ldots,2r\}$.
Acting by $\omega_j$ on $e^{\boldsymbol{\lambda}}$ using that $\nord\, h_i(m)h_i(n)\,\nord e^{\boldsymbol{\lambda}}=0$ unless $\max\{m,n\}\le r-\frac{1}{2}$ in which case $\min\{m,n\}=(m+n)-\max\{m,n\}\ge r-(r-\frac{1}{2})=\frac{1}{2}$, we get
\begin{equation}
\omega_j e^{\boldsymbol{\lambda}} = \frac{1}{2}\sum_{\substack{m,n\in\N+\frac{1}{2}\\ m+n=j-1}} (\lambda_m,\lambda_n)e^{\boldsymbol{\lambda}}
\end{equation}
In particular for $j=2r$ we have $m+n=2r-1$ and $m,n\le r-\frac{1}{2}$ in nonzero terms hence $m=n=r-\frac{1}{2}$ which means that \[\omega_{2r}e^{\boldsymbol{\lambda}}=\frac{1}{2} (\la_{r-\frac{1}{2}},\la_{r-\frac{1}{2}})e^{\boldsymbol{\lambda}} \neq 0\]
since $\la_{r-1/2}\neq 0$.
This proves the claim.
\end{proof}

\begin{Remark}
Note that $e^{\boldsymbol{\lambda}}$ is an eigenvector for $\omega_{2r+1}$ with eigenvalue $\frac{1}{2} (\la_r,\la_r)$ which is nonzero. Therefore the weak $M(1,\boldsymbol{\lambda})$ is not $\N$-graded.

On the other hand,
as noted in Example \ref{ex:r0case}, if $\boldsymbol{\lambda}=(\la_0,0,0,\ldots)$ then  $M(1,\boldsymbol{\lambda})$ is the level $1$ Verma $\widehat{\mathfrak{h}}$-module of highest weight $\la_0$, which is $\N$-graded.
\end{Remark}

\section{On the Whittaker type of $M(1,\boldsymbol{\la})$ and $M(1,\boldsymbol{\la})(\theta)$}

In the last section we proved that $M(1,\boldsymbol{\lambda})$ and $M(1,\boldsymbol{\lambda})$ give rise to Whittaker modules over the Virasoro algebra. In this section we show the map is surjective and calculate the fibers. That is, given a Whittaker type $\zeta$ for $\omega$, we find all possible $\boldsymbol{\lambda}$ giving rise to the Whittaker module of type $\zeta$ for the Virasoro algebra. Our results generalize the rank one case from \cite{T2}.

First we treat the untwisted case.

\begin{Proposition}
Let $r$ be a non-negative integer.
Define a function $\Phi=\Phi_r$ as follows:
\begin{equation}
\Phi: \mathfrak{h}^r\times(\mathfrak{h}\setminus\{0\})\to \C^r\times \C^\times
\end{equation}
given by
\begin{equation}\label{eq:lambda-system}
\Phi(\la_0,\la_1,\ldots,\la_r)=(\zeta_{r+1},\zeta_{r+2},\ldots,\zeta_{2r+1}),
\end{equation}
where
\begin{equation}
\zeta_i=\frac{1}{2}\sum_{\substack{0\le m,n\le r\\ m+n=i-1}} (\la_m,\la_n),\qquad i=r+1,r+2,\ldots,2r+1.
\end{equation}
Then $\Phi$ is surjective and for any $\zeta=(\zeta_{r+1},\zeta_{r+2},\ldots,\zeta_{2r+1})\in\C^r\times\C^\times$, the fiber $\Phi^{-1}(\{\zeta\})$ is is isomorphic to the variety  $S^{\ell-1}_{\C}\times \C^{(\ell-1)r}$ where
\[S^{\ell-1}_{\C}=\{(z_1,z_2,\ldots,z_\ell)\in \C^\ell\mid z_1^2+z_2^2+\cdots+z_\ell^2=1\}.\]
\end{Proposition}

\begin{proof}
Taking $i=2r+1, 2r, 2r-1, \ldots, r+1$, the system \eqref{eq:lambda-system} written out is
\[
\left\{\begin{aligned}
(\la_r,\la_r) &= 2\zeta_{2r+1} \\
(\la_{r-1},\la_r)+(\la_r,\la_{r-1}) &= 2\zeta_{2r} \\
(\la_{r-2},\la_r)+(\la_{r-1},\la_{r-1})+(\la_r,\la_{r-2}) &= 2\zeta_{2r-1} \\
&\;\;\vdots \\
(\la_0,\la_r)+(\la_1,\la_{r-1})+\cdots+(\la_r,\la_0) &= 2\zeta_{r+1}
\end{aligned}\right.\]
which is equivalent to
\[
\left\{\begin{aligned}
(\la_r,\la_r) &= 2\zeta_{2r+1}\\
(\la_{r-1},\la_r) &= \zeta_{2r}\\
(\la_{r-2},\la_r) &= \zeta_{2r-1} - \tfrac{1}{2}(\la_{r-1},\la_{r-1})\\
&\;\;\vdots \\
(\la_0,\la_r) &= \zeta_{r+1}-\tfrac{1}{2}\sum_{\substack{1\le m,n\le r-1\\ m+n=r}} (\la_m,\la_n)
\end{aligned}\right.\]
Letting $\la_{ij}=(\la_i,h_j)$ denote the coordinates of the vector $\la_i\in\mathfrak{h}$ in the basis $\{h_j\}_{j=1}^\ell$, the first equation is equivalent to the quadratic equation $\la_{r1}^2+\cdots+\la_{r\ell}^2 = 2\zeta_{2r+1}$. Putting $z_j=\frac{\la_{rj}}{(2\zeta_{2r+1})^{1/2}}$ (for some choice of square root of the nonzero complex number $2\zeta_{2r+1}$) shows that the $\la_r$ are in bijection with $S_\C^{\ell-1}$. The remaining equations are affine equations which have $(\ell-1)$-dimensional affine solution spaces at each of the $r$ steps, since $\la_r\neq 0$. This proves that $\Phi$ is surjective and that each fiber $\Phi^{-1}(\{\zeta\})$ is in bijection with the affine variety $S_\C^{\ell-1}\times\C^{(\ell-1)r}$.
\end{proof}

Similarly in the twisted case we have the following result.

\begin{Proposition}
Let $r$ be a positive integer.
Define a function $\Psi=\Psi_r$ as follows:
\begin{equation}
\Psi: \mathfrak{h}^{r-1}\times(\mathfrak{h}\setminus\{0\})\to \C^{r-1}\times \C^\times
\end{equation}
given by
\begin{equation}\label{eq:lambda-system}
\Phi(\la_{1/2},\la_{3/2},\ldots,\la_{r-1/2})=(\zeta_{r+1},\zeta_{r+2},\ldots,\zeta_{2r}),
\end{equation}
where
\begin{equation}
\zeta_i=\frac{1}{2}\sum_{\substack{m,n\in\{\frac{1}{2},\frac{3}{2},\ldots,r-\frac{1}{2}\}\\ m+n=i-1}} (\la_m,\la_n),\qquad i=r+1,r+2,\ldots,2r.
\end{equation}
Then $\Psi$ is surjective and for any $\zeta=(\zeta_{r+1},\zeta_{r+2},\ldots,\zeta_{2r})\in\C^{r-1}\times\C^\times$, the fiber $\Psi^{-1}(\{\zeta\})$ is isomorphic to the variety $S^{\ell-1}_{\C}\times \C^{(\ell-1)(r-1)}$ where
\[S^{l-1}_{\C}=\{(z_1,z_2,\ldots,z_l)\in \C^l\mid z_1^2+z_2^2+\cdots+z_l^2=1\}.\]

\end{Proposition}

\begin{proof}
Taking $i=2r, 2r-1, 2r-1, \ldots, r+1$, the system \eqref{eq:lambda-system} written out is
\[
\left\{\begin{aligned}
(\la_{r-1/2},\la_{r-1/2}) &= 2\zeta_{2r} \\
(\la_{r-3/2},\la_{r-1/2})+(\la_{r-1/2},\la_{r-3/2}) &= 2\zeta_{2r-1} \\
(\la_{r-5/2},\la_{r-1/2})+(\la_{r-3/2},\la_{r-3/2})+(\la_{r-1/2},\la_{r-5/2}) &= 2\zeta_{2r-2} \\
&\;\;\vdots \\
(\la_{1/2},\la_{r-1/2})+(\la_{3/2},\la_{r-3/2})+\cdots+(\la_{r-1/2},\la_{1/2}) &= 2\zeta_{r+1}
\end{aligned}\right.\]
which is equivalent to
\[
\left\{\begin{aligned}
(\la_{r-1/2},\la_{r-1/2}) &= 2\zeta_{2r}\\
(\la_{r-3/2},\la_{r-1/2}) &= \zeta_{2r-1}\\
(\la_{r-5/2},\la_{r-1/2}) &= \zeta_{2r-2} - \tfrac{1}{2}(\la_{r-3/2},\la_{r-3/2})\\
&\;\;\vdots \\
(\la_{1/2},\la_{r-1/2}) &= \zeta_{r+1}-\tfrac{1}{2}\sum_{\substack{m,n\in\frac{3}{2}+\N\\ m+n=r}} (\la_m,\la_n)
\end{aligned}\right.\]
The remaining argument is identical to the untwisted case.
\end{proof}

\bibliographystyle{siam}

\end{document}